\documentclass[a4paper, oneside]{amsart}

\usepackage[T1]{fontenc}
\usepackage[utf8]{inputenc}
\usepackage{lmodern}
\usepackage[english]{babel}

\usepackage{amsmath, amsthm, amssymb, amsaddr}

\theoremstyle{plain}
\newtheorem{theorem}{Theorem}[section]
\newtheorem{corollary}[theorem]{Corollary}
\newtheorem{lemma}[theorem]{Lemma}
\newtheorem{proposition}[theorem]{Proposition}

\theoremstyle{definition}
\newtheorem{definition}[theorem]{Definition}
\newtheorem{example}[theorem]{Example}
\newtheorem{question}[theorem]{Question}

\theoremstyle{remark}

\newtheorem{remark}[theorem]{Remark}

\newcommand{\cl}{\overline}
\DeclareMathOperator{\cll}{cl}

\newcommand{\B}{\mathcal{B}}
\newcommand{\U}{\mathcal{U}}
\newcommand{\V}{\mathcal{V}}

\newcommand{\R}{\mathbb{R}}
\DeclareMathOperator{\co}{co}

\title{Topological Properties of the Space of Convex Minimal Usco Maps}
\author{Ľubica Holá and Branislav Novotný}
\address{Mathematical Institute, Slovak Academy of Sciences, Štefánikova 49, SK-814 73 Bratislava, Slovakia}
\email{lubica.hola@mat.savba.sk, branislav.novotny@mat.savba.sk}
\keywords{convex minimal usco maps, complete metrizability, (upper) Vietoris topology, (strong) Choquet game}
\subjclass[2010]{54C60, 54C35, 54E50}
\begin{document}

\begin{abstract}
	Let $X$ be a Tychonoff space and $MC(X)$ be the space of convex minimal usco maps with values in $\R$, the space of real numbers. Such set-valued maps are important in the study of subdifferentials of convex functions. Using the strong Choquet game we prove complete metrizability of $MC(X)$ with the upper Vietoris topology. If $X$ is normal, elements of $MC(X)$ can be approximated in the Vietoris topology by continuous functions. We also study first countability, second countability and other properties of the upper Vietoris topology on $MC(X)$.
\end{abstract}
	\maketitle
\section{Introduction}
	Convex usco maps are interesting because they describe common features of maximal monotone operators, of the convex subdifferential and of the Clarke generalized gradient. Convex minimal usco maps are very important in functional analysis \cite{Borwein1997}, \cite{Phelps1993}, where differentiability of Lipschitz functions is deduced by their Clarke subdifferentials being convex minimal usco maps. Convex minimal usco maps also appear in the study of weak Asplund spaces \cite{Fabi1997}, \cite{Moors2006}, optimization \cite{Coban1989} and in the study of differentiability of Lipschitz functions \cite{Borwein1991}, \cite{Moors1995}.

	Also a classical problem of approximations of relations by continuous functions leads to the study of convex usco maps \cite{Cellin1970}, \cite{Hola2005}, \cite{Hola2007a}, \cite{Hola2007}. For this problem, let $X$ be a Hausdorff space, let $C(X)$ be the space of all continuous real-valued functions defined on $X$ and let $CL(X \times\R)$ be the hyperspace of all nonempty closed subsets of $X \times\R$, where $\R$ is the space of real numbers. It is known (see \cite{Beer1988},\cite{Hola1987},\cite{Hola1992}) that if $X$ is a locally connected, locally compact metric space without isolated points and $F \in  CL(X \times\R)$, then $F$ can be approximated by continuous functions in the Hausdorff metric if and only if $F$ is the graph of a convex usco map. The fundamental result needed to prove the above theorem is due to Cellina \cite{Cellin1970}.
	
\section{Preliminaries}
	Throughout this paper we will denote general topological spaces by capital letters $X,Y,Z$ and we will ask of them to be at least Hausdorff and non-empty. The symbol $CL(X)$ denotes the space of all closed nonempty subsets of $X$. On this space we will consider the \emph{Vietoris topology}, denoted by $\tau_V$, which is generated by the following subbase:
	\[\{V^+; V\text{ is an open subset of }X\}\cup\{V^-; V\text{ is an open subset of }X\},\]
	where
	\[V^+=\{A\in CL(X); A\subseteq V\}\text{ and }V^-=\{A\in CL(X); A\cap V\not=\emptyset\}.\]
	The first part, namely $\{V^+; V\text{ is an open subset of }X\}$ is the base of a topology called the \emph{upper Vietoris topology}, denoted by $\tau_V^+$ and it will be of the main interest for us as explained later. Analogically the other part generates the \emph{lower Vietoris topology} $\tau_V^-$. An usual notation for the base sets of the Vietoris topology is
	\[[W_1,...,W_n]=W_1^-\cap...\cap W_n^-\cap\left(\bigcup_{k=1}^nW_k\right)^+.\]
	
	We will not distinguish between set-valued \emph{maps} (multifunctions) and their graphs. We will say that a set-valued map $F$ is usc if it is upper semicontinuous at every $x\in X$ (see \cite{Engelking1977} for a definition). We will say that it is usco (resp. cusco) if it is usc and for every $x\in X$ the set $F(x)$ is a non-empty compact set (resp. a non-empty compact and convex set). We say that $F$ is minimal usco (resp. cusco) if it is usco (resp. cusco) and whenever $G\subseteq F$ is usco (resp. cusco) then $G=F$.
	
	Notice that in the introduction we have used term \emph{convex (minimal) usco} rather than \emph{(minimal) cusco}. Their meanings are the same. First one is used more often in the literature than second one, but throughout this paper we will stick to the second shorter term. 
	
	We will be interested in the following spaces:
	\begin{align*}
		L(X,Y)&:=\{F\subseteq X\times Y; F\text{ is cusco}\},\\
		L_0(X,Y)&:=\{F\subseteq X\times Y; F\text{ is cusco and for every isolated }x,\ F(x)\text{ is singleton}\},\\
		MC(X,Y)&:=\{F\subseteq X\times Y; F\text{ is minimal cusco}\},\\
		C(X,Y)&:=\{f:X\to Y; f\text{ is continuous}\}.
	\end{align*}
	Denote by $\R$ the space of real numbers with the usual topology and denote 
	\[L(X):=L(X,\R),\ L_0(X):=L_0(X,\R),\ MC(X):=MC(X,\R),\ C(X):=C(X,\R).\]
	Observe that
	\[CL(X\times Y)\supseteq L(X,Y)\supseteq L_0(X,Y)\supseteq MC(X,Y)\supseteq C(X,Y),\]
	so all of these spaces can inherit both the Vietoris and the upper Vietoris topology from $CL(X\times Y)$.
	\begin{remark}
		Note that when we will work with the space e.g. $L_0(X)$ then instead of writing $W^+\cap L_0(X)$ resp. $[W_1,W_2]\cap L_0(X)$ we will write only $W^+$ resp. $[W_1,W_2]$. The reader can easily figure out the complete expression from the context.
	\end{remark}
	Notice that if $X$ is regular then also $X\times\R$ is regular. Thus we have that $(CL(X\times\R),\tau_V)$ is Hausdorff and so are $L(X)$, $L_0(X)$, $MC(X)$ and $C(X)$ with $\tau_V$. The situation for $\tau_V^+$ is more complicated in general and we will address it later.
	
	We will be interested in completeness a countability properties of these spaces. For the cardinal invariants that are needed we refer to the book of Juhász \cite{Juhasz1980}. Denote by $\omega_0$  the infinite countable cardinal; i.e. the set $\{0,1,2,3,...\}$ and by $\omega_1$ the first uncountable cardinal.
	
	Since we will also work with normal and perfectly normal spaces, we will use the following results. For definitions of upper (resp. lower) semicontinuous functions we refer to \cite[Engelking]{Engelking1977}. We will denote them as u.s.c. (resp. l.s.c.).
	
	\begin{proposition}[{\cite[1.7.15 (b)]{Engelking1977}}]\label{prop:norm}
		A $T_1$ space $X$ is normal iff for every $f,h:X\to R$ such that $f\le h$, $f$ is u.s.c. and $h$ is l.s.c. there is $g\in C(X)$ such that $f\le g\le h$.
	\end{proposition}
	\begin{proposition}[{\cite[1.7.15 (c)]{Engelking1977}}]\label{prop:perf}
		A $T_1$ space $X$ is perfectly normal iff for every l.s.c. (u.s.c.) function $f:X\to R$ there is a sequence $(f_n)$ such that $f_n\in C(X)$ and $f_n\nearrow f\ (f_n\searrow f)$.
	\end{proposition}

	\begin{definition}[\cite{Neubrunn1988}]
		Let $X,Y$ be topological spaces. A function $f:X\to Y$ is quasicontinuous at a point $x\in X$ if for every open neighborhood $V$ of $f(x)$ and every open neighborhood $U$ of $x$ there is a non-empty open set $G\subseteq U$ (not necessarily containing $x$) such that for every $x'\in G$ it holds $f(x')\in V$.
	\end{definition}
	Quasicontinuity was first defined by Kempisty in \cite{Kempisty1932} for real functions of real variables.
	
	We will also need the following generalization of continuity for a function with a dense domain, i.e. densely defined function.
	\begin{definition}[\cite{Lechicki1990}]
		Let $X,Y$ be topological spaces and let $A$ be a dense subset of $X$. A function $f:A\to Y$ is said to be \emph{subcontinuous at $x\in X$} iff for every net $(x_\lambda)\subseteq A$ s.t. $x_\lambda\to x$, the net $(f(x_\lambda))$ has a cluster point. If $f$ is subcontinuous at every $x\in X$ it is \emph{subcontinuous}.
	\end{definition}
	Subcontinuity for functions with a full domain i.e. $A=X$ was defined by Fuller in \cite{Fuller1968}.
	\begin{theorem}[{\cite[2.4, 2.6]{Hola2014}}]\label{thm:characterization}
		Let $X$ be a topological space and $F\subseteq X\times\R$ be a set-valued map. The following are equivalent:
		\begin{enumerate}
			\item $F$ is minimal cusco;
			\item $F$ has a quasicontinuous subcontinuous selection $f:X\to\R$ s.t. $\co[\cl{f}(x)]=F(x)$ for every $x\in X$;
			\item there is a dense subset $A$ of $X$ and a quasicontinuous (on $A$) subcontinuous (on $X$) function $f:A\to\R$ s.t. $\co[\cl{f}(x)]=F(x)$ for every $x\in X$.
		\end{enumerate}
	\end{theorem}
	For a set-valued map $F\subseteq X\times Y$ denote \[S(F)=\{x\in X; F(x)\text{ is a singleton}\},\] and for a function $f:X\to Y$ denote \[C(f)=\{x\in X; f\text{ is continuous at }x\}.\]
	\begin{theorem}\label{thm:baire}
		Let $X$ be a Baire space and $F\in MC(X)$. Then $S(F)$ is a dense $G_\delta$-subset of $X$.
	\end{theorem}
	\begin{proof}
		By the virtue of Theorem \ref{thm:characterization} we have a quasicontinuous subcontinuous function $f:X\to\R$ s.t. $\co[\cl{f}(x)]=F(x)$. It is easy to see that for every $x\in S(F)$ (since $F$ is usc at $x$) $f$ has to be continuous at $x$; i.e. $S(F)\subseteq C(f)$.
		
		Now suppose that $x\in C(f)$. Then one can verify that $\cl{f}(x)=\{f(x)\}$; thus $x\in S(F)$. 
		
		We have proven that $S(F)=C(f)$ and now since $X$ is a Baire space, $\R$ is a metric space and $f$ is quasicontinuous, we have that $C(f)$ is a dense $G_\delta$-subset of $X$, see \cite{Neubrunn1988}.
	\end{proof}

\section{Countability properties of $C(X)$}
	
	In this section we mostly collect some useful facts about the space $(C(X),\tau_V)$ which are needed for the study of the space $MC(X)$.  Denote by $\tau_U$ the \emph{topology of uniform convergence} on $C(X)$; i.e. the topology generated by the supremum metric \[d(f,g)=\sup_{x\in X}|f(x)-g(x)|.\] Note that on $C(X)$ we have $\tau_U\subseteq\tau_V$.
	
	The following theorem is a well known fact, see eg. \cite{Hola2015}.
	\begin{theorem}\label{thm:VplusV}
		Topologies $\tau_V^+$ and $\tau_V$ coincide on $C(X)$.
	\end{theorem}
	
	\begin{theorem}[{\cite[Theorem 5.1]{Hola2015}}]\label{thm:CXfirst}
		Let $X$ be a Tychonoff space. Then the following are equivalent:
		\begin{enumerate}
			\item $(C(X),\tau_V)$ is metrizable;
			\item $(C(X),\tau_V)$ is first countable; \label{eq:char}
			\item $(C(X),\tau_V)$ has a countable $\pi-$character; \label{eq:pichar}
			\item $(C(X),\tau_V)$ is a Fréchet space;
			\item $(C(X),\tau_V)$ is sequential;
			\item $(C(X),\tau_V)$ has a countable tightness;
			\item $X$ is countably compact; \label{eq:X}
			\item $\tau_V=\tau_U$. \label{eq:eq}
		\end{enumerate}
	\end{theorem}
	\begin{proof}
		In addition to the \cite[Theorem 5.1]{Hola2015} we have \eqref{eq:pichar} and \eqref{eq:eq}. The equality of \eqref{eq:char} and \eqref{eq:pichar} follows from the fact that character and $\pi-$character coincide for topological groups and $(C(X),\tau_V)$ is a topological group with respect to the addition. The equality of \eqref{eq:X} and \eqref{eq:eq} is proven in \cite{Hansard1970}.
	\end{proof}
	The following theorem is a generalization of \cite[Proposition 1.2]{Hola2013}.
	\begin{theorem}\label{thm:CXsecond}
	Let $X$ be a Tychonoff space. Then the following are equivalent:
		\begin{enumerate}
			\item $(C(X),\tau_V)$ is second countable; \label{eq:weight}
			\item $(C(X),\tau_V)$ is separable;
			\item $(C(X),\tau_V)$ satisfies countable chain condition; \label{eq:cell}
			\item $(C(X),\tau_V)$ is Lindelöff;
			\item $(C(X),\tau_V)$ has a countable $\pi-$base;
			\item $(C(X),\tau_V)$ has a countable network;
			\item $(C(X),\tau_V)$ has a countable spread;
			\item $(C(X),\tau_V)$ has a countable extent; \label{eq:extent}
			\item $X$ is compact and metrizable. \label{eq:XX}
		\end{enumerate}
	\end{theorem}
	\begin{proof}
		Recall that all considered cardinal invariants are less than or equal to the weight of $\tau_V$ and grater or equal to either cellularity or extent. Therefore we need to prove only \eqref{eq:XX}$\Rightarrow$\eqref{eq:weight}, \eqref{eq:cell}$\Rightarrow$\eqref{eq:XX} and \eqref{eq:extent}$\Rightarrow$\eqref{eq:XX}. First two implications are proven in \cite[Proposition 1.2]{Hola2013}. For the last suppose that $(C(X),\tau_V)$ has a countable extent then $(C(X),\tau_U)$ has a countable extent and therefore $X$ is compact and metrizable.
	\end{proof}
	\begin{theorem}[{\cite[Corollary 3.3]{Hola2015}}]\label{thm:Vbaire}
		The space $(C(X),\tau_V)$ is Baire.
	\end{theorem}

\section{Relationship of C(X) and MC(X)}
	We will need the following useful lemma.
	\begin{lemma}[{\cite[Lemma 4.1]{Hola2008}}]\label{lem:connected}
		Let $F\in L(X)$ and let $W\subseteq X\times\R$ be an open set containing $F$. Then there exists an open set $G\subseteq X\times R$ such that $G(x)$ is connected for each $x \in X$ and $F \subseteq G  \subseteq W$.
	\end{lemma}
	Note that in the preceding Lemma $F\in[G]\subseteq [W]$ and $[G],[W]\in\tau_V^+$. The following result improves the one from \cite[Lemma 4.1]{Hola2005}, where $X$ is assumed to be countably paracompact and normal. We use some ideas from its proof.
	\begin{theorem}\label{thm:upper}
		Let $X$ be a  normal space, $F\in L(X)$ and $F\in [W]\in\tau_V^+$. Then $[W]\cap C(X)\not=\emptyset$.
	\end{theorem}
	\begin{proof}
		By Lemma \ref{lem:connected} we can assume that $W(x)$ is connected for every $x \in X$. 
		Let $\phi : X \times  (- \pi/2, \pi/2)  \to X \times R$ be the homeomorphism defined by $\phi(x, t) = (x, \tan t)$, let $F_0=\phi^{-1}(F)$ and $W_0=\phi^{-1}(W)$. Define $f_1,f_2,h_1,h_2:X\to R$ by
		\[\begin{matrix}
		h_1(x)=\inf F_0(x),& f_1(x)=\inf W_0(x),\\
		h_2(x)=\sup W_0(x),& f_2(x)=\sup F_0(x).
		\end{matrix}\]
		One can verify that $f_1,f_2$ are u.s.c. and $h_1,h_2$ are l.s.c. and $f_1<h_1\le f_2<h_2$.  By Proposition \ref{prop:norm} there are $g_1,g_2\in C(X)$ such that
		\[f_1\le g_1\le h_1\le f_2\le g_2\le h_2.\]
		Put $g_0=\frac{g_1+g_2}{2}$ and observe that 
		\[-\frac{\pi}{2}\le f_1<g_0<h_2\le \frac{\pi}{2},\]
		therefore $g=\tan\circ\, g_0\in C(X)$ and $g(x)\in W(x)$ for all $x\in X$.
	\end{proof}
	\begin{corollary}\label{cor:dense}
		Let $X$ be a normal space. Then $C(X)$ is dense in $(L(X),\tau_V^+)$ and consequently also in $(L_0(X),\tau_V^+)$ and $(MC(X),\tau_V^+)$.
	\end{corollary}
	From Theorems \ref{thm:Vbaire} and  \ref{thm:VplusV} we know that for a Tychonoff space $X$, the space $(C(X),\tau_V^+)$ is a Baire space and from this we immediately have the following.
	\begin{corollary}\label{cor:densebaire}
		Let $X$ be a normal space. Then $(L(X),\tau_V^+)$, $(L_0(X),\tau_V^+)$ and $(MC(X),\tau_V^+)$ are Baire spaces.
	\end{corollary}
	\begin{theorem}\label{thm:whole}
		Let $X$ be a normal space, $F\in L_0(X)$ and $F\in [W_1,...,W_n]\in\tau_V$. Then $[W_1,...,W_n]\cap C(X)\not=\emptyset$.
	\end{theorem}
	\begin{proof}
		By Lemma \ref{lem:connected} we can choose open $W_0\subseteq \bigcup_{n=1}^nW_k$ such that $F\in W_0^+$ and $W_0(x)$ is connected for all $x\in X$. By Theorem \ref{thm:upper} there is $f\in C(X)\cap W_0^+$. Now for $k=1,..,n$ choose $(x_k,t_k)\in W_0\cap W_k$. Since $F(x)$ is a singleton for every isolated $x$, we may assume that all $x_k$ are distinct. Now choose open intervals $V_k$ for $k=1,..,n$ such that $\{t_k,f(x_k)\}\subseteq V_k\subseteq\cl{V}_k\subseteq W_0(x_k)$ and choose open sets $U_k$ such that $f(U_k)\subseteq V_k$, $U_k\times\cl{V}_k\subseteq W_0$ and $U_k$ are pairwise disjoint. Then from Tietze's extension theorem there is $g\in C(X)$ such that $f$ and $g$ agree on $X\setminus (U_1\cup..\cup U_n)$, $g(x_k)=t_k$ and $g(U_k)\subseteq\cl{V}_k$. Therefore $g\in [W_1,...,W_n]\cap C(X)$.
	\end{proof}
	Since minimal cusco maps are single valued at isolated points of $X$ we have the following corollary.
	\begin{corollary}\label{cor:denseV}
		Let $X$ be a normal space. Then $C(X)$ is dense in $(L_0(X),\tau_V)$ and consequently also in $(MC(X),\tau_V)$.
	\end{corollary}
	Simmilarly to \ref{cor:densebaire} we have the following.
	\begin{corollary}
		Let $X$ be a normal space. Then $(L_0(X),\tau_V)$ and $(MC(X),\tau_V)$ are Baire spaces.
	\end{corollary}
	The following example shows that the normality of $X$ is essential in Theorems \ref{thm:upper} and \ref{thm:whole}. It is similar to \cite[Example 9]{Artico2008}.
	\begin{example}
		Let $X$ be a space of nonlimit ordinals less than $\omega_1$, together with $\omega_1$, equipped with the order topology. Let $X=X_0\cup X_1\cup\{\omega_1\},$ where $X_0,X_1$ are disjoint uncountable sets not containing $\omega_1$. Let $Y=[0,\omega_0]$ with the order topology and $Y_0=[0,\omega_0)$. Let $Z=X\times Y\setminus\{(\omega_1,\omega_0)\}$. Let $F\subseteq Z\times[0,1]$ be a  map defined by
		\[F(x,y)=\begin{cases}
		\{0\},&x\in X_0,y\in Y,\\
		\{1\},&x\in X_1,y\in Y,\\
		[0,1],&x=\omega_1,y\in Y_0.
		\end{cases}\]
		One can verify that $Z$ is not normal and $F$ is cusco. We will show that there is an open set $W\subseteq Z\times\R$ such that $F\subseteq W$ and $W^+\cap C(Z)=\emptyset.$
	
		Put $W=\Big(X_0\times Y\times(-1,1/3)\Big)\cup\Big(X_1\times Y\times(2/3,2)\Big)\cup\Big(X\times Y_0\times(-1,2)\Big)$. Suppose that there is $f\in C(Z)\cap W^+.$
	
		Since $f$ is continuous then for every fixed $y<\omega_0$ we have that $f(x,y)$ is eventually constant, i.e. there is $x_y$ such that for every $x>x_y$ holds $f(x,y)=f(\omega_1,y)$. Since $cf(\omega_1)=\omega_1$ we have that $\hat{x}:=\sup \{x_y;y<\omega_0\}<\omega_1$. 
	
		Since $f\in W^+$, for every $x_0\in X_0$ and $x_1\in X_1$ we have $f(x_0,\omega_0)<1/3$ and $f(x_1,\omega_0)>2/3$. Since $\hat{x}$ is countable, we can choose $\hat{x}_0\in X_0,\ \hat{x}_0>\hat{x}$ and $\hat{x}_1\in X_1,\ \hat{x}_1>\hat{x}$. Since $f$ is continuous, there is $n_0<\omega_0$ such that for every $y>n_0$ we have $f(\hat{x}_0,y)<1/3$ and $f(\hat{x}_1,y)>2/3$, which contradicts the fact that both values are supposed to be equal to $f(\omega_1,y)$.
	\end{example}
	Notice that in the above proofs (especially in the proof of Theorem \ref{thm:whole}) the "nice" part of $\tau_V$ is actually $\tau_V^+$ and $\tau_V^-$ is a nuisance that we have to deal with. From Theorem \ref{thm:VplusV} we know that $\tau_V=\tau_V^+$ for $C(X)$. Since for a Baire space $X$, spaces $MC(X)$ and $C(X)$ are quite close, in the sense that any $F\in MC(X)$ is single-valued on a residual set, it is natural to ask if this is true also for $MC(X)$. Unfortunately the answer is no as the following example shows.
	\begin{example}
		Let $X=R\setminus\left\{\frac{1}{n};n\in\omega_0\setminus\{0\}\right\}$ and let $F,f_n\subseteq X\times\R$ for $n\in\omega_0\setminus\{0\}$ be defined by
		\[F(x)=\begin{cases}
		\{0\}& x<0,\\ [0,1]&x=0,\\ \{1\} &x>0;
		\end{cases}\qquad
		f_n(x)=\begin{cases}
		0& x<\frac{1}{n},\\ 1 &x>\frac{1}{n}.
		\end{cases}	\]
		We can easily check that $X$ is Baire; $F,f_n (n\in\omega_0\setminus\{0\})\in MC(X)$ and $f_n\to F$ in $\tau_V^+$. Now put $W=X\times (0,1)$. Since $F\in W^-$ and $f_n\not\in W^-$ for $n\in\omega_0\setminus\{0\}$, we have that $f_n\not\to F$ in $\tau_V^-$.
	\end{example}
	Nevertheless the topology $\tau_V^+$ may be better suited for the study of the space $MC(X)$.
	\begin{proposition}\label{prop:hausdorff}
		For a Baire space $X$, the space $(MC(X),\tau_V^+)$ is Hausdorff.
	\end{proposition}
	\begin{proof}
		Let $F,G\in MC(X)$ s.t. $F\not=G$. From Theorem \ref{thm:baire} we have that $S(F)\cap S(G)$ is a dense $G_\delta-$set i.e. there is $x\in S(F)\cap S(G)$ s.t. $F(x)\not= G(x)$; otherwise they would agree on a dense subset and by Theorem \ref{thm:characterization}, they would be equal. Choose open disjoint $U,V\subseteq \R$ such that $F(x)\in U$ and $G(x)\in V$. Sets $U_0=[X\times U]\cup[(X\setminus\{x\})\times\R]$ and $V_0=[X\times V]\cup[(X\setminus\{x\})\times\R]$ are open, disjoint and $F\subseteq U_0,\ G\subseteq V_0$.
	\end{proof}
	Suppose again that $X$ is a Baire space. The fact that minimal cusco maps are single-valued on a residual set allows us to define an addition on the set $MC(X)$ in a natural way. Notice that for $F,G\in MC(X)$ we have that $S(F)\cap S(G)$ is a dense $G_\delta-$set and \[f=F\restriction_{S(F)\cap S(G)},\quad g=G\restriction_{S(F)\cap S(G)}\]
	are continuous functions which are also subcontinuous (by \cite[3.3 and 3.8]{Hola2012} every selection of an usco map is subcontinuous and therefore its densely defined restriction is also subcontinuous). We can define \[(F+G)(x)=\co[\cl{(f+g)}(x)].\]
	From Theorem \ref{thm:characterization} we have that $F+G\in MC(X)$.
	However this addition is not continuous in $\tau_V^+$ nor in $\tau_V$ as the following example shows.
	\begin{example}
		Let $F,G,f_n,g_n (n\in\omega_0\setminus\{0\})\subseteq \R\times \R$ be defined by
		\[F(x)=\begin{cases}
		\{0\}& x<0,\\ [0,1]&x=0,\\ \{1\} &x>0;
		\end{cases}\qquad
		G(x)=\begin{cases}
		\{0\}& x<0,\\ [-1,0]&x=0,\\ \{-1\} &x>0;
		\end{cases}\]
		\[f_n(x)=\begin{cases}
		0& x<-\frac{1}{n},\\ nx+1 & x\in\left[-\frac{1}{n},0\right],\\1 &x>0;
		\end{cases}\qquad
		g_n(x)=\begin{cases}
		0& x<0,\\ -nx & x\in\left[0,\frac{1}{n}\right],\\-1 &x>\frac{1}{n}.
		\end{cases}\]
		One can easily see that $F+G=0$, $f_n\to F$, $g_n\to G$ in $\tau_V$ and obviously they converge also in $\tau_V^+$. But $f_n+g_n\not\to 0=F+G$ in $\tau_V^+$ and therefore neither in $\tau_V$.
	\end{example}
\section{Countability properties of $MC(X)$}
	\begin{theorem}\label{thm:chi}
		Let $X$ be a Tychonoff space. If any of the following spaces $(MC(X),\tau_V^+)$, $(L_0(X),\tau_V^+)$, $(L(X),\tau_V^+)$, $(MC(X),\tau_V)$, $(L_0(X),\tau_V)$ or $(L(X),\tau_V)$ is first countable, then $X$ is countably compact.
	\end{theorem}
	\begin{proof}
		First countability of any of these spaces implies first countability of $(C(X),\tau_V^+)$ and from Theorem \ref{thm:CXfirst} we have that $X$ has to be countably compact.
	\end{proof}
	\begin{proposition}
		Let $X$ be a normal space. Then TFAE:
		\begin{enumerate}
			\item $(L(X),\tau_V^+)$ is first countable,
			\item $X$ is countably compact and perfectly normal.
		\end{enumerate}
	\end{proposition}
	\begin{proof}
		We start with $(2)\Rightarrow (1)$. This is actually stated in \cite[Theorem 4.7]{Hola2008} without proof so for reader's convenience we provide one. Fix an arbitrary $F\in L(X)$ and put $f(x)=\inf F(x)$ and $g(x)=\sup F(x)$. Observe that $f$ is l.s.c. and $g$ is u.s.c. Since $X$ is perfectly normal, there are, by virtue of Proposition \ref{prop:perf}, continuous functions $f_n\!\nearrow f$ and $g_n\!\searrow g$. Define open sets $W_n\subseteq X\times\R$ by $W_n(x)=(f_n(x),g_n(x))$ and we have that $F=\bigcap_{n\in\omega_0}W_n$. Since $X\times\R$ is normal, for every $n\in\omega_0$ there is an open $V_n$ such that $F\subseteq V_n\subseteq\cl{V}_n\subseteq W_n$. We will prove that $\{V_n;n\in\omega_0\}$ is a local base at $F$ by showing that for every open $W\supseteq F$ there is $n\in\omega_0$ such that $V_n\subseteq W$. Suppose it is not, then there is an open $W\supseteq F$ such that for every $n\in\omega_0$ there is $(x_n,y_n)\in V_n\setminus W$. Since $X$ is countably compact, there are $\alpha_n,\beta_n\in\R$ such that $\cl{V}_n\subseteq X\times[\alpha_n,\beta_n]$ i.e. $\cl{V}_n\setminus W$ is countably compact and therefore $(x_n,y_n)$ has a cluster point $(x,y)\in \bigcap_{n\in\omega_0}\cl{V}_n\setminus W=\emptyset$, contrary to supposition.
		
		Now we prove $(1)\Rightarrow (2)$. From Theorem \ref{thm:chi} we have that $X$ is countably compact. To prove that it is also perfectly normal we will use Proposition \ref{prop:perf}. Take an arbitrary u.s.c. function $f:X\to R$. For every $n\in\omega_0$ put $f_n(x)=\max\{f(x),-n\}$ and $F_n(x)=[-n,f_n(x)]$. Observe that $F_n\in L(X)$, $f_n$ is u.s.c and $f_n\to f$ as $n\to\infty$. There is $\{W_{n,m}^+;m\in\omega_0\}$ a local base at $F_n$. Put $g_{n,m}(x)=\sup W_{n,m}(x)$. Since $g_{n,m}$ is l.s.c. there is by Proposition \ref{prop:norm} a function $f_{n,m}\in C(X)$ such that $f_n\le f_{n,m}\le g_{n,m}$. It is easy to see that $f_{n,m}\to f_n$ as $m\to\infty$. Put $g_n=\min\{f_{p,q};p,q=1..n\}$. One can verify that $g_n\searrow f$ as $n\to\infty$.
	\end{proof}
	\begin{corollary}\label{cor:firstcount}
		Let $X$ be a countably compact, perfectly normal space. Then $(L_0(X),\tau_V^+)$ and $(MC(X),\tau_V^+)$ are first countable.
	\end{corollary}
	\begin{theorem}\label{thm:spread}
		Let $X$ be a Tychonoff space. The following are equivalent:
		\begin{enumerate}
			\item $(MC(X),\tau_V^+)$ is second countable;
			\item $(MC(X),\tau_V^+)$ has a countable network;
			\item $(MC(X),\tau_V^+)$ has a countable spread;
			\item $(MC(X),\tau_V^+)$ is Lindelöff;
			\item $(MC(X),\tau_V^+)$ has a countable extent;
			\item $X$ is compact and metrizable.
		\end{enumerate}
	\end{theorem}
	\begin{proof}
		$(1)\Rightarrow(2)$, $(2)\Rightarrow(3)$, $(2)\Rightarrow(4)$, $(3)\Rightarrow(5)$ and $(4)\Rightarrow(5)$ are clear. To prove $(5)\Rightarrow(6)$ suppose $(MC(X),\tau_V^+)$ has a countable extent. Then $(C(X),\tau_V^+)$ has a countable extent and by Theorem \ref{thm:CXsecond} space $X$ must be compact and metrizable.  Finally to prove $(6)\Rightarrow(1)$ assume that $X$ is compact and metrizable. Then $X$ has a countable base. It is easy to verify, that $MC(X)\subseteq K(X\times\R)$, where $K(X\times\R)$ is the space of all non-empty compact subsets of $X\times\R$. Since $X\times\R$ has a countable base, we know that $(K(X\times\R),\tau_V^+)$ has also a countable base, see \cite{Michael1951}.
	\end{proof}
	Using Corollary \ref{cor:dense} we give a simpler proof to a slight generalization of \cite[Proposition 5.5]{Hola2008}
	\begin{theorem}\label{thm:ccc}
		Let $X$ be a normal space. If any of the spaces $(L(X),\tau_V)$, $(L(X),\tau_V^+)$, $(L_0(X),\tau_V)$, $(L_0(X),\tau_V^+)$, $(MC(X),\tau_V)$, $(MC(X),\tau_V^+)$ satisfies countable chain condition then $X$ is compact and metrizable.
	\end{theorem}
	\begin{proof}
		First note that if $(L(X),\tau_V)$ satisfies ccc then also $(L(X),\tau_V^+)$ satisfies ccc. Since $C(X)$ is dense in all of the mentioned spaces (except for $(L(X),\tau_V$) then also $(C(X),\tau_V)$ satisfies ccc. From Theorem \ref{thm:CXsecond} we have that $X$ must be compact and metrizable.
	\end{proof}
	The following theorem extends the results of Theorem \ref{thm:spread} in the case of normal $X$.
	\begin{theorem}
		Let $X$ be a normal space. The following are equivalent:
		\begin{enumerate}
			\item $(MC(X),\tau_V^+)$ has a countable $\pi-$base;
			\item $(MC(X),\tau_V^+)$ is separable;
			\item $(MC(X),\tau_V^+)$ satisfies a countable chain condition;
			\item $X$ is compact and metrizable.
		\end{enumerate}
	\end{theorem}
	\begin{proof}
		$(1)\Rightarrow(2)$ and $(2)\Rightarrow(3)$ are clear. $(3)\Rightarrow(4)$ follows from Theorem \ref{thm:ccc} and $(4)\Rightarrow(1)$ follows from Theorem \ref{thm:spread}.
	\end{proof}
	For $f,g:X\to\R$, $f<g$ denote \[M_{f,g}=\{(x,y)\in X\times R; f(x)<y<g(x)\}.\]
	\begin{proposition}[{\cite[Proposition 4.4]{Hola2008}}]\label{prop:base}
		Let $X$ be a countably paracompact normal space and $D$ be a dense subset of $(C(X),\tau_V^+)$. Then
		\[\B=\{M_{f,g}^+;f,g\in D, f<g\}\]
		is a base of $(L(X),\tau_V^+)$.
	\end{proposition}
	\begin{lemma}\label{lem:reg}
		Let $X$ be a space and $f,g\in C(X)$, such that $f<g$. Then in $(MC(X),\tau_V^+)$ we have that $\cll(M_{f,g}^+)\subseteq \cl M_{f,g}^+$.
	\end{lemma}
	\begin{proof}
		Put $W=M_{f,g}$. Observe that $\cl W(x)=\cl{W(x)}=[f(x),g(x)]$. Choose an arbitrary $F\in\cll(W^+)$. We will prove that for every $x\in X$ holds $F(x)\cap \cl W(x)\not=\emptyset$ by contradiction. Suppose that there is $x_0\in X$ such that $F(x_0)\subseteq \cl W(x_0)^C$. Then $F\subseteq V:=(X\times\R) \setminus \left(\{x_0\}\times \cl{W}(x_0)\right)$, where $V$ is an open subset of $X\times\R$. Since $F\in V^+\cap\cll(W^+)$ then $V^+\cap W^+\not=\emptyset$, a contradiction.
		
		We have proved that $F\cap\cl W$ has nonempty values and since $F\in MC(X)$ and $\cl W$ is closed and convex valued upper semicontinuous  map, then $F\cap\cl W$ is cusco. From the minimality of $F$ we have that $F\cap\cl W=F$, thus $F\subseteq\cl W$, i.e. $F\in\cl W^+$.
	\end{proof}
	\begin{theorem}\label{thm:reg}
		Let $X$ be a countably paracompact normal space. Then $(MC(X),\tau_V^+)$ is regular.
	\end{theorem}
	\begin{proof}
		Choose an arbitrary $F\in MC(X)$ and open $W\subseteq X\times R$ such that $F\in W^+$. Since $X$ is countably paracompact and normal, then $X\times R$ is normal (see \cite{McCoy2000}) and there is an open $W_0$ such that $F\subseteq W_0\subseteq \cl{W}_0\subseteq W$. From Proposition \ref{prop:base} we have that $\B$ is a base of $\tau_V^+$ so there is $M_{f,g}$ such that $F\in M_{f,g}^+\subseteq W_0^+$ and from Lemma \ref{lem:reg} it follows that
		\[F\in M_{f,g}^+\subseteq \cll(M_{f,g}^+)\subseteq\cl{M}_{f,g}^+\subseteq\cl{W}_0^+\subseteq W^+.\]
	\end{proof}
	\begin{definition}[{a more rigorous definition can be found in \cite[Definition 8.10]{Kechris1995}}]
		Let $X$ be a topological space. A \emph{Choquet game} $G_X$ of $X$ is defined as follows. Players I and II take turns in playing open sets. Player I starts with $U_0$, player II follows with $V_0\subseteq U_0$ and then again player I plays $U_1\subseteq V_0$ and so on, creating an infinite sequence called a \emph{run of the game}. Player I wins if $\bigcap_{n\in\omega_0} U_n=\emptyset$, otherwise player II wins. A \emph{strategy} for player I is a rule that assigns exactly one $U_n$ to any $U_0,V_0,...,V_{n-1}$. A strategy is \emph{winning} if player I wins every run consistent with this strategy. A \emph{tactics} for player I is a rule that assigns exactly one $U_n$ for every $V_{n-1}$. A \emph{winning} tactics is defined analogically to a wining strategy. Also a (winning) strategy and tactics for player II are defined analogically.
		
		A space $X$ is called \emph{Choquet} if player II has a winning strategy.
	\end{definition}
	Note that every winning tactics defines a winning strategy, but the converse is not true.
	\begin{definition}[{\cite[Definition 8.14]{Kechris1995}}]
		Let $X$ be a topological space. A \emph{strong Choquet game} $G^s_X$ for $X$ is defined similarly to $G_X$, but player I plays a pair $(x_n,U_n)$ with $x_n\in U_n$ instead of just $U_n$ and player II has to choose $V_n$ so that $x_n\in V_n\subseteq U_n$.
		
		A space $X$ is called \emph{strong Choquet} if player II has a winning strategy.
	\end{definition}
	Note that every strong Choquet space is Choquet, but the converse is not true.
	\begin{theorem}\label{thm:stCh}
		Let $X$ be a countably paracompact normal space. Then $(MC(X),\tau_V^+)$ is strong Choquet.
	\end{theorem}
	\begin{proof}
		We will present a winning tactics for player II in the strong Choquet game. Suppose that player I has chosen $F_k\in MC(X)$ and $\U_k\in\tau_V^+$ such that $F_k\in\U_k\subseteq\V_{k-1},$ where $\V_{k-1}$ is the choice of player II from the previous turn. From the proof of Theorem \ref{thm:reg} there are $f_k,g_k\in C(X)$ such that
		\[F_k\in M_{f_k,g_k}^+\subseteq \cl{M}_{f_k,g_k}^+\subseteq\U_k.\]
		Player II chooses $\V_k=M_{f_k,g_k}^+$. Now we have
		\[f_1<\ldots <f_k<\ldots<g_k<\ldots<g_1.\]
		Put $F(x):=\bigcap_{k\in\omega_0}[f_k(x),g_k(x)]\not=\emptyset$. Since $F$ is an intersection of a decreasing system of cusco maps, we have that  $F$ is cusco, therefore it contains a minimal cusco map $F^*$. Since $F^*\in\bigcap_{k\in\omega_0}\U_k$ we have described a wining tactics.
	\end{proof}
	Recall that a second countable, completely metrizable topological space is called \emph{Polish}.
	\begin{theorem}\label{thm:polish}
		Let $X$ be a Tychonoff space. Then $(MC(X),\tau_V^+)$ is Polish iff the space $X$ is compact and metrizable.
	\end{theorem}
	\begin{proof}
		If $(MC(X),\tau_V^+)$ is Polish, then it is second countable and from Theorem \ref{thm:spread} we have that $X$ must be compact and metrizable.
		
		By \cite[Theorem 8.18]{Kechris1995} a nonempty, second countable, topological space is Polish iff it is $T_1$, regular and strong Choquet. By Theorem \ref{thm:spread} the space $(MC(X),\tau_V^+)$ is second countable, by Theorem \ref{thm:reg} it is regular, by Theorem \ref{thm:stCh} it is strong Choquet and by Proposition \ref{prop:hausdorff} it is Hausdorff.
	\end{proof}
	\begin{lemma}\label{lem:almostbase}
		Let $\U$ be an open set in the space $(L(X),\tau_V)$ such that there is $h\in\U\cap C(X)$. Then there are $f,g\in C(X)$ such that $h\in M_{f,g}^+\subseteq\cl{M}_{f,g}^+\subseteq\U$.
	\end{lemma}
	\begin{proof}
		W.l.o.g. we can suppose that $\U=W_0^+\cap W_1^-\cap .. \cap W_n^-$, where $W_k$ are open subsets of $X\times\R$. Choose $x_1,..,x_n$ so that $(x_k,h(x_k))\in W_k$. Let $x_1,..,x_{n'}$ be all those x's which are distinct (rearranged if necessary). For $k=1,..,n'$ we can choose open sets $U_k$ and open intervals $J_k$ so that $U_k$ are pairwise disjoint and $(x_k,h(x_k))\in U_k\times J_k\subseteq W_k$. Put $\U'=W_0^+\cap(U_1\times J_1)^-\cap..\cap(U_{n'}\times J_{n'})^-$ and observe that $h\in\U'\subseteq\U$. Since by Theorem \ref{thm:VplusV} we have that $\tau_V=\tau_V^+$ on $C(X)$, then using Proposition \ref{prop:base} there are $f',g'\in C(X)$ such that \[h\in M_{f',g'}^+\cap C(X)\subseteq \U'\cap C(X).\] One can verify that for any $f,g\in C(X)$ such that $f'<f<h<g<g'$ we have that $h\in M_{f,g}^+\subseteq\cl{M}_{f,g}^+\subseteq\U'.$
	\end{proof}
	\begin{theorem}
		Let $X$ be a countably paracompact normal space. Then $(L_0(X),\tau_V)$ and  $(MC(X),\tau_V)$ are Choquet.
	\end{theorem}
	\begin{proof}
		We will present a winning tactics for player II in the Choquet game on the space $(MC(X),\tau_V)$. Suppose that player I has chosen $\U_k\in\tau_V$ such that $\U_k\subseteq\V_{k-1},$ where $\V_{k-1}$ is the choice of player II from the previous turn. Since by Corollary \ref{cor:denseV} the set $C(X)$ is dense in $(MC(X),\tau_V)$ there is $h_k\in C(X)\cap\U_k$. Using Lema \ref{lem:almostbase} there are $f_k,g_k\in C(X)$ such that 
		\[h_k\in M_{f_k,g_k}^+\subseteq\cl{M}_{f_k,g_k}^+\subseteq\U_k.\]
		Player II chooses $\V_k=M_{f_k,g_k}$. We have constructed
		\[f_1<\ldots <f_k<\ldots<g_k<\ldots<g_1.\]
		Similarly as in the proof of Theorem \ref{thm:stCh} put $F(x):=\bigcap_{k\in\omega_0}[f_k(x),g_k(x)]\not=\emptyset$. Since $F$ is an intersection of a decreasing system of cusco maps, we have that  $F$ is cusco, therefore it contains a minimal cusco map $F^*\in\bigcap_{k\in\omega_0}\U_k$, which concludes the proof for the space $(MC(X),\tau_V)$. The proof for the space $(L_0(X),\tau_V)$ is analogous.
	\end{proof}
	Even though the Theorem \ref{thm:polish} is a result about metrizability of the space $(MC(X),\tau_V^+)$ we do not have the answer to the following problem.
	\begin{question}
		Characterize spaces $X$ for which the space $(MC(X),\tau_V^+)$ or $(MC(X),\tau_V)$ is metrizable.
	\end{question}
	In the rest of this section we collect some miscellaneous results concerning the above question.
	
	Let us define a (possibly infinite valued) metric $L$ on $MC(X)$ by $L(F,G)=\sup\{H(F(x),G(x)); x\in X\},$ where $H$ is the Hausdorff metric generated by the usual metric on $\R$. Denote by $\tau_L$ the topology generated by $L$.
	\begin{theorem}
		Let $X$ be a countably compact space. Then $\tau_V\subseteq\tau_L$.
	\end{theorem}
	\begin{proof}
		We will prove that for every open $W\subseteq X\times R$ and every $F\in MC(X)$ such that $F\in W^+$ there is $\varepsilon>0$ such that $B^L_\varepsilon(F)\subseteq W^+$, where
		\[B^L_\varepsilon(F)=\{G\in MC(X);\ L(F,G)<\varepsilon\}.\]
		 Denote
		\[f_1(x)=\inf F(x),\quad  f_2(x)=\sup F(x)\]
		and observe that $f_1$ is l.s.c and $f_2$ is u.s.c. Using Proposition \ref{prop:base} there are $g_1,g_2\in C(X)$ such that $F\in M_{g_1,g_2}^+\subseteq W^+$. Functions $f_1-g_1$ and $g_2-f_2$ are positive valued and l.s.c. Since $X$ is countably compact, there is $\varepsilon>0$ such that $f_1-g_1>\varepsilon$ and $g_2-f_2>\varepsilon$. One can easily verify that $B^L_\varepsilon(F)$ is as desired.
		
		Now suppose $F\in W^-$, where $W$ is open. Choose $(x,y)\in F\cap W$. There is $\varepsilon>0$ such that $\{x\}\times (y-\varepsilon,y+\varepsilon)\subseteq W$. Again, one can easily verify that $B^L_\varepsilon(F)\subseteq W^-$.
	\end{proof}
	
	\begin{definition}[\cite{Stover2009}]
		Topological space $Y$ is said to be weakly $\pi-$metrizable iff it has a $\sigma-$disjoint $\pi-$base.
	\end{definition}
	Recall that a $\sigma-$disjoint system is a countable union of families consisting of pairwise disjoint sets.
	\begin{theorem}
		Let $X$ be a countably compact perfectly normal space. Then $(MC(X),\tau_V^+)$ is weakly $\pi-$metrizable
	\end{theorem}
	\begin{proof}
		Since $X$ is countably compact, then $(C(X),\tau_V)=(C(X),\tau_V^+)$ is metrizable. And since $X$ is also perfectly normal, then by Corollary \ref{cor:firstcount} the space $(MC(X),\tau_V^+)$ is first countable Hausdorff space, which by Corollary \ref{cor:dense} contains a dense metrizable subspace. Thus by \cite[Theorem 2.6]{White1978} the space $(MC(X),\tau_V^+)$ has a $\sigma-$disjoint $\pi-$base; i.e. it is weakly $\pi-$metrizable.
	\end{proof}
	\textbf{Acknowledgements.} Both authors were supported by VEGA 2/0006/16.

\end{document}